\documentclass[11pt,noinfoline]{article}

\RequirePackage[amsthm,amsmath,natbib]{imsart}
\RequirePackage[OT1]{fontenc}
\RequirePackage{mymathbbol}
\usepackage{amsfonts,latexsym,amssymb}
\def\bb{\mathbb}
\usepackage{graphics,graphicx}

\startlocaldefs
\def\bb{\mathbb }
\def\hf{\widehat f}
\def\:{:\,}

\def\cal{\mathcal}
\def\a{\alpha}
\def\b{\beta}
\def\tsigma{{\widetilde\sigma}}

\def\det{{\mbox{\rm det}}}

\def\LKC{Lipschitz-Killing curvature}

\def\P{{\bb P}}
\def\E{{\bb E}}
\def\smallhalf{\mbox{ $\frac{1}{2}$}}

\def\e{{\varepsilon}}
\def\p{\varphi}
\def\real{{\bb{R}}}

\def\RN{\mathbb{R}^N}

\def\indic{1}

\def\definedas{\stackrel{\Delta}{=}}

\def\sas{$S\alpha S$}
\def\SaS{$S\alpha S$}

\newcommand{\sqbinom}[2]{\begin{bmatrix}#1 \\ #2 \end{bmatrix}}

\def\EC{Euler characteristic}

\def\cO{{\cal O}}

\def\qed{\hfill $\Box$\newline}

\newcommand{\bc}{\begin{center}}
\newcommand{\ec}{\end{center}}

\newcommand{\beq}{\begin{eqnarray}}
\newcommand{\eeq}{\end{eqnarray}}
\newcommand{\beqq}{\begin{eqnarray*}}
\newcommand{\eeqq}{\end{eqnarray*}}

\newcommand{\Rr}{\mathbb{R}}

\newcommand{\lips}{{\cal L}}

\newcommand{\Tube}{{\rm Tube}}

\def\text{\mbox}

\newcommand{\graff}{{\rm Graff}}

\newcommand{\vep}{\varepsilon}

\newtheorem{theorem}{Theorem}[section]
 
 \newtheorem{lemma}[theorem]{Lemma}
 
 \newtheorem{corollary}[theorem]{Corollary}

\endlocaldefs

\begin{document}
\begin{frontmatter}
\title{Excursion sets of stable random fields} \runtitle{Excursion
sets of stable random fields}
\begin{aug}
\author{\fnms{Robert J.} \snm{Adler}\thanksref{t1}\ead[label=e1]{robert@ieadler.technion.ac.il}
\ead[label=u1,url]{ie.technion.ac.il/Adler.phtml}}
\author{\fnms{Gennady}
  \snm{Samorodnitsky}\thanksref{t1,t2}\ead[label=e2]{gennady@orie.cornell.edu}
\ead[label=u2,url]{www.orie.cornell.edu/$\sim$gennady/} }
\and
\author{\fnms{Jonathan E.} \snm{Taylor}\thanksref{t1,t3}
\ead[label=e3]{taylor@dms.umontreal.ca}
\ead[label=u3,url]{www-stat.stanford.edu/$\sim$jtaylo/}}

\thankstext{t1}{Research supported in part by US-Israel Binational
Science Foundation, grant 2004064.}
\thankstext{t2}{Research supported in part by NSA grant MSPF-05G-049 and ARO
grant W911NF-07-1-0078}
\thankstext{t3}{Research supported in part by NSF grant DMS-0405970,
and the Natural Sciences and Engineering Research Council  of Canada.}
\runauthor{Adler, Samorodnitsky and Taylor}
\affiliation{Technion, Cornell, Stanford and Montreal}
\end{aug}
\begin{abstract}
Studying the geometry generated by Gaussian and Gaussian-related random fields
via their excursion sets is now a well developed and well understood subject.
The purely non-Gaussian scenario has, however, not been studied at all.
In this paper we look at three classes of stable random fields, and obtain
asymptotic formulae for the mean values of various geometric characteristics
of their excursion sets over high levels.

While the formulae are asymptotic, they contain enough information to show that not
only do stable random fields exhibit geometric behaviour very different 
from that of
Gaussian fields, but they also differ significantly among themselves.

\end{abstract}

\end{frontmatter}

\section{Introduction}

We are interested in the structure of the sample paths of certain smooth
stable random fields $f\:\RN\to\real$,
$N\geq 1$.
We shall study these through certain  geometric properties of their
{\it excursion sets}
  \beq
\label{excursion-set}
A_u \ \equiv\ A_u(f,M)\ \definedas \
\left\{ t\in M\: f(t)\geq u\right\},
\eeq
where $M\subset\RN$
and $u\in\real$.

Excursion sets have been widely studied for Gaussian and Gaussian
related random fields. Their  applications appear in
disciplines as widespread as astrophysics and medical imaging, where
they have been also been used in a variety of hypothesis testing
situations A good introductory reference to the applications is
still Keith Worsley's exposition \cite{WOR97} although
\cite{ARF}, when ready, will have a lot more detail. On the  more
theoretical side, where excursion sets are seen to generate
an elegant geometric structure, our basic reference will be
the recent monograph \cite{RFG}.

The Gaussian and Gaussian related scenarios allow for the
development of explicit formulae for the expectations of many
of the geometrical quantifiers of  excursion sets.
Unfortunately, one cannot expect that the same
will occur in the stable case, for
which our basic reference will be  \cite{ST94}. Here, explicit formulae
for even the marginal densities of $f$ are unknown, although much
is  known about their asymptotics. Indeed,
when we began this research we expected to be able to find little
beyond some asymptotic formulae  relating to the excursion sets
generated by stable fields that might mimic the Gaussian ones, much as we did
for level crossings of stable processes on the real line, in
 \cite{AdlerSam97} and \cite{AST93}.

What we found turned out to be far more interesting. It is well known
that the structure of stable processes is far more complicated
than that of their
Gaussian counterparts. For example, whereas in the Gaussian case
many stationary processes have both a moving average  representation
(with respect to white noise) as well as a harmonisable representation,
in the stable case moving average and harmonisable processes
belong to quite distinct families. These differences are well
understood in terms of mathematical structure, but it turns out that
the sample path distinctions between different classes of stable
processes  become significantly highlighted by looking
at the excursion sets that they generate, primarily in the
multi-dimensional setting.

Thus, this paper has two aims. One is to provide explicit, albeit
asymptotic, formulae for the expectations of the Euler characteristics
(defined below) and other geometric quantifiers of excursion sets generated by stable random
fields. These, we believe, will find immediate application  in random
field modelling using these processes. The second is to better
understand the differences between various stable random
fields via their excursion sets.

The remainder of the
paper is structured as follows. In the next section we shall
 define the Euler characteristic
and related \LKC s. In Section
\ref{section:Gauss} we give a
description of the main result from the Gaussian theory, which is a precise formula
for the expected value of the \LKC s of Gaussian excursion sets. With this in
hand, in Sections \ref{section:expect}--\ref{main:wavelimited}  we start with
the new results, for sub-Gaussian, harmonisable, and concatenated-harmonisable random
fields, deriving asymptotic formulas for the expected values of the Euler characteristics
of their excursion sets. In Section \ref{section:isotropic} we show how to lift these
results to all the \LKC s, at least for isotropic fields.
A technical appendix  completes the paper.

\section{Euler characteristics and \LKC s}
\label{section:Euler}

Throughout this paper our parameter sets $M$ will be taken to be convex subets of
$\RN$, and, for much of it, we shall restrict ourselves to
 $N$-dimensional rectangles. In these cases we shall always write $T$ rather than $M$,
where $T$ is given by
\beq
\label{T:definition}
T\definedas \prod_{j=1}^N [0,T_i].
\eeq

For convex sets in $\RN$, and for excursion sets of smooth functions
defined over them, there
are $N+1$ functionals which describe their geometry. These are known under a
variety of names, including
{\it Minkowski functionals, intrinsic volumes,  quermassintegrals},
and {\it \LKC s}, being related to one another by
differences in the way they are ordered or  normalized. For consistency with
\cite{RFG}, which we shall use heavily when we wish to cite a result
without proof, we shall work with \LKC s. Perhaps the easiest way to
define these is via Steiner's formula, a classic result of integral geometry.

To state Steiner's formula, let  $M$ be a convex set of dimension
$N$, sitting in $\real^{N'}$, where $N'\geq N$. (e.g.\ $M$ is a
one-dimensional curve in $\real^2$.) The {\it tube of radius $\rho$}
around $M$ is defined to  be
\beqq
\Tube (M,\rho) = \{x \in \real^{N'}\: \inf_{y\in M} |x-y| \leq \rho\}.
\eeqq

With $\lambda_N$ denoting  Lebesgue measure in $\RN$, Steiner's
formula states that there is an exact polynomial expansion of order
$N$ for the $\lambda_{N'}$ measure of $\Tube (M,\rho)$ given by
 \beq
\label{geometry:steiner:formula}
\lambda_{N'}
\left(\Tube (M,\rho)\right)  =
\sum_{j=0}^N \omega_{N'-j} \rho^{N'-j} \lips_j(M),
\eeq
where
\beqq
\omega_j =
\frac{\pi^{j/2}}{\Gamma\bigl(\frac{j}{2} + 1 \bigr)}
\eeqq
is the volume of the unit ball in $\Rr^j$. The numbers $\lips_j(M)$
are  the {\it \LKC s}  of $M$.

The $\lips_j$  scale nicely, in the sense that, for $\lambda >0$,
$\lips_j(\lambda M)  =  \lambda^j\lips_j(M)$. As is obvious from
\eqref{geometry:steiner:formula}, $\lips_N(M)$ measures the volume of
$M$, $\lips_{N-1}(M)$ is related to its surface measure, etc. The last
one, $\lips_0(M)$, is the {\it Euler characteristic} of $M$ which,
since we shall use it very often, we also denote by $\p (M)$.
Of all of the $\lips_j(M)$, the Euler characteristic is the only one that does not
change under smooth deformations of $M$. If $M$ is one dimensional, then $\p (M)$ counts
the number of connected components in $M$. If $M$ is two dimensional, then it counts the
number of connected components minus the number of holes. In three dimensions, it counts
the number of connected components, minus the number of `handles', plus the number of
holes.

We note, for later usage, that the \LKC s of the $N$-rectangle
 \eqref{T:definition} are given by
\beq
\label{geometry:LK-rectangle:equation}
\lips_j(T) = \lips_j\Big(\prod_{i=1}^{N} [0,T_i]\Big)
 =  \sum  T_{i_1}\dots T_{i_j},
\eeq
where the sum is taken over the ${N \choose j}$ distinct choices of
subscripts $i_1,\dots,i_j$.

There is another way to write \eqref{geometry:LK-rectangle:equation}, which will be
useful later on. Let
 $\cal O_j\equiv \cal O_j(T)$ denote the collection of the
${N \choose j}$ $j$-dimensional facets of
$T$ which contain the origin. (Thus, for example, $\cal O_N$ is $T$ itself,
while $\cal O_1$ contains the $N$ one-dimensional edges of $T$ lying
on the positive axes $e_1,\dots,e_N$ of $\real^N$.) Furthermore,
if $J$ is a facet in $\cal O_j$, let $|J|$ denote
its $j$-dimensional Lebesgue measure. Then it is immediate from
\eqref{geometry:LK-rectangle:equation} that
\beq
\label{geometry:LK-rectangl2e:equation}
\lips_j(T)  =  \sum_{J\in\cal O_j} |J|.
\eeq

The \LKC s play a central r\^ole in much of integral geometry, but for the moment
we shall note only one of their properties, known as Hadwiger's theorem
\cite{H57}.
Suppose that we have a functional $\psi$ on compact convex sets that
is additive, in the sense that, if $A$, $B$ and $A\cup B$ are compact
convex, then
\beq
\psi(A \cup B) = \psi(A)+\psi(B)-\psi(A \cap B).
\label{G6.1.2}
\eeq
If it is also true that $\psi$ is invariant under rigid motions and
 continuous in the Hausdorff metric, then
there are  ($\psi$-dependent) constants such that
\beq
\label{geometry:additivity-func:equation}
\psi(A) =  \sum_{j=0}^N c_j \lips_j(A).
\eeq
Thus, studying intrinsic volumes is equivalent to studying a far wider class of functionals
on sets. Our aim is to study the \LKC s of excursion sets.

\section{Gaussian excursion set geometry}
\label{section:Gauss}
In this section we want to summarise some results about the excursion sets of
Gaussian random fields. There are two reasons for bringing these. The first
is that it gives us a basis to which to compare the results of this
paper for stable random fields,
 and the second is that, in all the cases that we shall consider in this
paper, the proof for the
stable case follows from the Gaussian one and a conditioning argument.

To state the main result for the Gaussian case, we need a little notation,
for which we now assume that $f$ is a mean zero, stationary,
 Gaussian random field on  $\RN$ with constant variance $\sigma^2$.
Assuming that $f$ is also almost surely $C^2$, we define the second order
spectral moments
\beqq
\lambda_{ij}=\E\left\{ \frac{\partial f(t)}{\partial t_i}
                 \frac{\partial f(t)}{\partial t_j} \right\}.
\eeqq
With $\{e_1,\dots,e_N\}$ denoting  the positive axes of $\RN$, suppose that $J\in\cal O_j$ is
a $j$-dimensional facet of a rectangle $T$. We write $\Lambda_J$ for the matrix
\beq
\label{LambdaJ}
\Lambda_J\ \definedas \ \left\{\lambda_{ij}\:i,j\in \sigma(J)\right\},
\eeq
where
\beqq
\sigma(J)\definedas \left\{j\:e_{j}\cap (J\setminus \{0\})\neq \emptyset\right\}.
\eeqq
Note that if $f$ is isotropic then there is a constant, which we shall write as
$\lambda_2$, such that
\beq
\label{lambda2}
\lambda_{ij}= \begin{cases}
\lambda_2& i=j,\\ 0 &i\neq j.
\end{cases}
\eeq

Next, we need the Hermite
polynomials
\beqq
H_n(x) =  n!\, \sum_{j=0}^{\lfloor n/2\rfloor}
\frac{(-1)^j x^{n-2j}}{j!\, (n-2j)!\, 2^j}, \qquad n\geq 0,\ x\in\real,
\eeqq
where $\lfloor a\rfloor$ is the largest integer  less than or equal to
$a$, and, for notational convenience, we define
\beqq
 H_{-1}(x) = \sqrt{2\pi}\Psi (x)e^{x^2/2},
\eeqq
where
\beqq
\Psi (x) \definedas
 \frac{1}{\sqrt{2\pi}}
\int_x^{\infty} e^{-u^2/2} \,du.
\eeqq
We also adopt the notation
\beqq
 \rho_n(u)
 = {(2 \pi)^{-(n+1)/2}}\, H_{n-1}(u)\,e^{-u^2/2}, \qquad
n \geq 0.
\eeqq
The following is a combination of Theorem 11.7.2 and the discussion in Section 11.8
of \cite{RFG}.

\begin{theorem}
 \label{rgeometry:EEC:theorem}
Let $f$ be a zero mean, stationary Gaussian
field on  a  $N$-rectangle
$T$ with  variance $\sigma^2$ and a.s.\ $C^2$ sample paths, 
and such that the joint
distribution of $f$ and its first and second derivatives at each  point $t\in T$
 is non-degenerate.
Suppose that the joint modulus of continuity
$\omega(\eta)$ of all the second order partial derivatives of $f$ satisfies
\beq
\label{modcont:equn}
\P\left\{\omega(\eta)>\varepsilon\right\}  =  o\left(\eta^N\right)
\qquad\text{as $\eta\downarrow 0$}
\eeq
for all $\varepsilon >0$. Then the mean value of the Euler characteristic
 of  its excursion set is  given by
\beq
\label{rgeometry:EEC:equation}
   \E\left\{\p \left(A_u(f,T)\right)\right\}  =
\sum_{n=0}^{N}
\sum_{J\in \cO_n}
\frac{|J|\,|\det \left(\Lambda_J\right)|^{1/2}}{\sigma^{n}}\,
\rho_{n}\left(\frac{u}{\sigma}\right).
\eeq
Furthermore, if $f$ is isotropic and  $M$ compact and convex, then,
for all $0\leq j\leq N$,
\beq
\label{rgeometry:EEC-isotropic-LK:equation} \qquad
\E\left\{\lips_j \left(A_u(f,M)\right)\right\} =
\sum_{n=0}^{N-j} \sqbinom{N}{n} \cal L_{n+j}(M)
\rho_{n}\left(\frac{u}{\sigma}\right)
\left(\frac{\lambda_2}{\sigma^2}\right)^{(n+j)/2},
\eeq
where $\lambda_2$ is as in \eqref{lambda2} and
\beqq
\sqbinom{N}{j}=\binom{N}{j} \frac{\omega_N}{\omega_{N-j} \; \omega_j}.
\eeqq

\end{theorem}

In fact, results of this nature hold in much wider
 generality, when $M$ is a general stratified manifold $M$ and
$f$ is neither stationary not isotropic. We refer the interested reader to
 Chapters 12 and 13 of  \cite{RFG}.

One observation that follows from  \eqref{rgeometry:EEC:equation}
and \eqref{rgeometry:EEC-isotropic-LK:equation}
comes by rewriting the sums
as power series in $u$, from which one immediately sees that
the leading order terms, of orders $u^{N-1}$ and $u^{N-j-1}$ respectively, are
 associated with
the volume of $T$ or $M$, while lower order terms are associated  the other
\LKC s.
This observation will be important when it comes to understanding the
stable case.

Finally, although it will not be important in what follows, we note
that the condition \eqref{modcont:equn} is an extremely mild one for
$C^2$ Gaussian processes, and is easily checked from the covariance
function of the process.

\section{Sub-Gaussian fields}
\label{section:expect}
With the Gaussian case behind us, we shall now look at what is
probably the simplest of all stable random fields, the sub-Gaussian
ones. Despite their simplicity, we shall see that their behavior is
already very different from the Gaussian case.

To define these processes, we
let $g$ be a Gaussian random field on $M$, and, for some $\a\in (0,2)$, let
 $X$ be a  $S_{\a/2}(\sigma_\alpha ,1,0)$
random variable independent of $g$ (see \cite{ST94} for notation), where
\beqq
\sigma_\alpha \definedas \cos(\pi\a/4)^{2/\a}.
\eeqq
Thus $X$ is a positive strictly $\a/2$-stable random variable with
Laplace transform
\beqq
\E\{e^{-t X} \} = e^{-t^{\a/2}},\quad t>0.
\eeqq
Taking $X$ independent of $g$ and setting
\beqq
f(t)  =  X^{1/2}g(t).
\eeqq
defines a {\it sub-Gaussian} random field.

To state our first result we need some notation.
For any functions $a,b\:\real\to\real$ we write
\beqq
a \asymp b \quad
\iff \quad  \lim_{u\to \infty}\frac{a(u)}{b(u)} =1.
\eeqq

\begin{theorem}
\label{theorem-subgaussian}
Let $g$ be a zero mean, stationary Gaussian field on the $N$-rectangle
$T$, satisfying the
conditions of Theorem \ref{rgeometry:EEC:theorem}. Denote its variance by
$\sigma^2_g$ and the matrices of its  second order spectral moments by
$\Lambda_J$, as in \eqref{LambdaJ}.
Let $f$, as above,  be the sub-Gaussian field $X^{1/2}g$.
Then
\beq
\label{main:subgaussian:formula}
\E\left\{\p\left(A_u(f,T)\right)\right\}
\asymp
u^{-\a}\left(K_{0} +\sum_{n=1}^N K_n
\sum_{J\in \cal O_n} |J| |\Lambda_J|^{1/2}\right),
\eeq
where we write  $ |\Lambda_J|$ for $|\det (\Lambda_J)  |$,
\beqq
K_0 =  \frac{ 2^{-1+\a/2} \sigma_g^\a \Gamma\left(\frac{\a +1}{2}\right)  }{\sqrt{\pi}\Gamma(1-\frac{\a}{2})},
\eeqq
and, for $n\geq 1$,
\beqq
K_n=
\frac{\a 2^{-1+\a/2} \sigma_g^\a \Gamma\left(\frac{\a +1}{2}\right)  }{\Gamma(1-\frac{\a}{2})}
 \frac{(n-1)!}{\pi^{(n+1)/2}\sigma_g^n}
\sum_{j=0}^{\lfloor (n-1)/2\rfloor} \frac{(-1)^j\Gamma\left(\frac{\a +n-1-2j}{2}
\right)}{2^{2j+1}j!(n-1-2j)!}.
\eeqq
If $g$ (and so $f$) is isotropic, with second spectral moment $\lambda_2$ (cf.\ \eqref{lambda2}) and
$M$ is a compact convex domain, then
\beq
\label{main:subgaussian-convex:formula}
\E\left\{\p\left(A_u(f,M)\right)\right\}
\asymp
u^{-\a}\sum_{n=0}^N K_n\lambda_2^{n/2}\lips_n(M).
\eeq
\end{theorem}

Under isotropy, a corresponding result holds for the expected \LKC s of  other orders as
well. We shall look at this later, in Section \ref{section:isotropic}.

Before proving this theorem, we shall take a moment to see how very
different it is from the purely Gaussian case, despite the fact that
the excursion sets of $f$ and $g$ are very simply related by the  fact
that
\beqq
A_u(f,T)=A_{uX^{-1/2}}(g,T).
\eeqq
Consider the case $N=1$, when the theorem relates to a $f$  defined over the interval $[0,T]$. Then
\beqq
\p\left(A_u(f,[0,T])\right) = \indic_{f(0)\geq u} + C_u(f,T),
\eeqq
where $C_u(f,T)$ is the number of upcrossings of the level $u$ by $f$ in $[0,T]$.
Taking expectations, we see that the two terms that appear in
\eqref{main:subgaussian:formula} correspond to (asymptotics for)
 $\E\{\indic_{f(0)\geq u}\}=
\P\{f(0)\geq u\}$ and
\beqq
\E\left\{C_u(f,T)\right\} \asymp u^{-\a}\,
\frac{2^{-1+\a /2}\Gamma\left(1+\frac{\a}{2}\right) (\lambda_{11})^{1/2}T}{\pi
\Gamma\left(1-\frac{\a}{2}\right) \sigma_g^{(1-\a )/2}},
\eeqq
where the asymptotics here come from either substitution in \eqref{main:subgaussian:formula} or Theorem 3.2 of \cite{AST93}, which studies one dimensional level crossings.
Note that both terms  -- i.e.\ the ``boundary'' and ``interior'' terms --
 have the same asymptotics, of the form $u^{-\a}$.
This is also true when looking at the case of  general $N$ in
 \eqref{main:subgaussian:formula}, in that facets of $T$ of all dimensions contribute
to the asymptotics,  and this is probably the most interesting aspect
of the result.

Recall that in the Gaussian case we saw that in an expansion of the mean
Euler characteristic the leading term
involved only the volume of $T$, with the surface area affecting only the
second and later terms, etc. In the sub-Gaussian case, however, it is clear from
 \eqref{main:subgaussian:formula} that the full geometry of $T$ affects the first
term of any such expansion.

A heuristic explanation for this is easy to find. In the Gaussian case, if the
level $u$ is high, the excursion set will, with high probability, contain only a
small set which will be unlikely to intersect the boundary of $T$. Hence, only
volume terms appear in the highest order term when expanding
$ \E\left\{\p\left(A_u(f,T)\right)\right\}$. However, in the notation of the
theorem, an excursion set
of the sub-Gaussian $f$ at the level $u$ has the same geometry as an excursion
set of the Gaussian $g$ at the level $u/\sqrt{X}$. Although $u$ may be
large, the most likely reason for $f$ to reach this level is that
$X$ also be   large and that  $u/\sqrt{X}$ is roughly 
$O(1)$. This being the case, $A_{u\sqrt{X}}(g,T)$ is a
``typical'' rather than ``rare'' excursion set for $g$, and so has a reasonable
probability of meeting the lower dimensional facets of $T$.
Thus it is not surprising that
they  contribute to  \eqref{main:subgaussian:formula}.

Despite the (hopefully) convincing tone of these heuristics, the proof follows
a different, and purely analytic, route.

\vspace*{0.15truein}
\noindent {\bf Proof of Theorem  \ref{theorem-subgaussian}.}
We shall start with the general, non-isotropic case, and $T$ a rectangle.

As noted above,
$A_u(f,T)=$ $A_{uX^{-1/2}}(g,T)$. Thus it is immediate  that
 $\p(A_u(f,T))$ is well defined since this the same is true of
$\p(A_u(g,T))$, for every $u$. Conditioning on $X$, we would now like
to claim that 
\beq
\label{eq:0}
\E\left\{  \p \left(A_u(f,T)\right)\right\}
&=& \E\left\{\E\left\{  \p \left(A_u(f,T)\right)\right\} \big| X\right\},
\eeq
and then use the Gaussian Theorem  \ref{rgeometry:EEC:theorem} to
compute the inner expectation.

However, to justify this we need to establish two facts.
The first is that the conditioned process, $f\big| X$, which is
clearly Gaussian, satisfies all the conditions of Theorem
\ref{rgeometry:EEC:theorem}. This is trivial, since $g$ is assumed to
satisfy these conditions and  $f\big|X$ is no more than a constant
multiple of $g$. 

The trickier problem is that to apply the iterated expectation in
\eqref{eq:0} we need to know, {\it a priori}, that the absolute moment 
$\E\left\{ \left| \p \left(A_u(f,T)\right)\right|\right\}$ 
is finite. Since this is technical we shall leave it to Lemma
\ref{appendix:subgauss} in the appendix, and for the moment progress
assuming that it is true.

Then, applying  Theorem  \ref{rgeometry:EEC:theorem} to \eqref{eq:0}, we have
\beqq
&&\E\left\{  \p \left(A_u(f,T)\right)\right\}\\
&&\qquad\qquad = \E\left\{\E\left\{  \p \left(A_{u/\sqrt{X}}(g,T)\right)\,\Big|\,  X\right\}
\right\}\\
\nonumber
&&\qquad \qquad = \sum_{n=0}^{N}\E\left\{H_{n-1}\left(\frac{u}{\sigma_g\sqrt{X}}\right)e^{-u^2/2\sigma_g^2X}\right\}
 \sum_{J\in \cal O_n}
\frac{|J| \,|\Lambda_J|^{1/2}}{(2\pi)^{(n+1)/2}\sigma_g^n}
.
\eeqq
To evaluate this triple sum (the third sum appears implicitly in the Hermite polynomials) we
need to consider  typical terms of the form
\beqq
\E\left\{ u^k (\sigma_g^2X)^{-k/2}e^{-u^2/2\sigma_g^2X}\right\}
 =  \left(\frac{u}{\sigma_g}\right)^k \E\left\{ X^{-k/2}e^{-u^2/2\sigma_g^2X}\right\}
\eeqq
and the one atypical term, coming from $n=0$ and $H_{-1}$, of the form
\beqq
\E\left\{ \Psi\left(\frac{u}{\sigma_g\sqrt{X}}\right)\right\}.
\eeqq
The asymptotics of these expressions is covered in Lemma \ref{l:tauberian},
which will be crucial to most of the  computations of this section.
Once we prove  Lemma \ref{l:tauberian},
 \eqref{main:subgaussian:formula} is the consequence of a little
 algebra, and so the first  part of the theorem is established.

As for the isotropic case, note that taking $M\equiv T$,
\eqref{main:subgaussian-convex:formula} follows  immediately from
\eqref{main:subgaussian:formula} on noting
\eqref{geometry:LK-rectangl2e:equation} and the fact that $\det
\Lambda_J=\lambda_2^{\dim J}$. To establish the argument for general e
convex compact $M$ we use this fact
and Hadwiger's result \eqref{geometry:additivity-func:equation}.

To this end, define a functional $\psi$ on convex compact sets
$M\subset\RN$ by setting 
\beqq
\psi(M) \definedas \E\left\{  \p \left(A_u(f,M)\right)\right\}.
\eeqq
It is immediate that $\psi$ is additive  (in the sense of \eqref{G6.1.2})
and, using the same general arguments as in the
appendix, by bounding  $ \p \left(A_u(f,M)\right)$  by the number 
of critical points of $f$ over $M$, that $\psi$ is also
continuous over convex compact sets. Furthermore, if $g$ is
isotropic, it follows that 
$\psi$ is also invariant under rigid motions. Thus Hadwiger's result
applies and we have 
 that there exist constants $c_j$ such that, for convex $M$,
\beqq
 \E\left\{  \p \left(A_u(f,M)\right)\right\} = \sum_{j=0}^N c_j \lips_j(M).
\eeqq
Since the $c_j$ are not dependent on $M$, the fact that
\eqref{main:subgaussian-convex:formula} holds for rectangles defines them, and this fact and
the above equation implies that \eqref{main:subgaussian-convex:formula}    holds for convex $M$ as well, as required.
\qed
\vspace*{0.15truein}

\begin{lemma} \label{l:tauberian}
 Let $X$ be  as in Theorem \ref{theorem-subgaussian}. Then,
for any $\beta>-\alpha$,
\beq  \label{e:tauberian}
 && \lim_{u\to \infty} u^{\alpha + \beta} \E\left\{ X^{-\beta/2}
e^{-u^2/2\sigma_g^2X}\right\} \\ && \qquad \qquad =
2^{(\a +\b -2)/2} \alpha C_{\alpha/2}\,\sigma_\a^{\alpha/2}\sigma_g^{\a +\b} \Gamma\left(\frac{\a +\beta}{2}\right).  \notag
\eeq
Furthermore,
\beq  \label{psi:tauberian}  \qquad
 \lim_{u\to \infty} u^{\alpha} \E\left\{ \Psi\left(\frac{u}{\sigma_g\sqrt{X}}\right)\right\}
&=& 2^{-1+\a/2} \pi^{-1/2} C_{\alpha/2}\,\sigma_\a^{\alpha/2}\sigma_g^{\a}
\Gamma\left(\frac{1+\a }{2}\right).
\eeq
\end{lemma}
\begin{proof}
The limit \eqref{e:tauberian} is a trivial consequence of
Lemma 2.2 of \cite{AdlerSam97}. As for \eqref{psi:tauberian}, let $G$ be a
standard normal variable. Then
\beqq
 \E\left\{ \Psi\left(\frac{u}{\sigma_g\sqrt{X}}\right)\right\}
&=& \smallhalf \P\left\{G^2 X> u^2/\sigma_g^2\right\}\\
&\asymp& \smallhalf \E\{|G|^\alpha\}  \P\left\{ X> u^2/\sigma_g^2\right\},
\eeqq
the last line following from a classic result of \cite{breiman:1965}.

Since $\E\{|G|^\a\} =2^{\a/2}\pi^{-1/2}\Gamma(\frac{1+\a}{2})$
and
$\P\{X>v\}\asymp C_{\a/2}\sigma_\a^{\a/2} v^{-\a/2} $   (see
 (3.7.2) and (1.2.8) of \cite{ST94}, respectively), the result now follows.
\qed
\end{proof}

%
%



\section{Harmonisable fields}
\label{main:harmonisable}
The sub-Gaussian random fields of the previous section provide an
interesting class of processes, in that they show that even a relatively
minor perturbation of the Gaussian scenario leads to quite different
behavior of the excursion sets. However, they do not represent a particularly
rich class of stable fields. A much richer  class of stable fields is
given by the  stationary,  symmetric, $\alpha$-stable 
 (\sas) harmonisable
ones. These are random fields possessing  a
spectral type representation of the form
\beq
\label{integralrep:equation}
f(t) = {\text{Re} }\Big\{ \int_{\RN} e^{it\omega}\,Z(d\omega)
\Big\},
\eeq
where $Z$ is a complex, \sas,
 Borel random measure on $\real^N$ with finite
control measure $\mu$. (See \cite{ST94} for details of this  and the
following representation.)

While \eqref{integralrep:equation} may explain from
where the terminology comes, there is an alternative representation that
will be much more useful for us, and which is given by
\beq
\label{sum:equation} \qquad
f(t) = \big(C_\alpha b_\alpha^{-1} \mu_0\big)^{1/\alpha}
        \sum_{k=1}^\infty \Gamma_k^{-1/\alpha}
       \left(G^{(1)}_k\cos(t\omega_k) + G^{(2)}_k\sin(t\omega_k)\right),
\eeq
where $t\in\real^N$ and  the product $t\omega_k$ is actually the inner product $\sum_{i=1}^N t(i)\omega_k(i)$.
The $\{G^{(i)}_k\}$, $i=1,2$, are independent sequences of i.i.d.\
standard normal variables. $\{\Gamma_k\}$ is the sequence of arrival
times of a unit rate Poisson process, $\{\omega_k\}$ is a
sequence of i.i.d.\ $\real^N$ valued
 random variables with probability measure
$\mu(\cdot)/\mu_0$ where
\beq
\label{mu0}
\mu_0\definedas \mu (\real^N)
\eeq
and  $\mu$ is the control measure for \eqref{integralrep:equation}. The four sequences are
independent of one another. The constants  $C_\alpha$
and $b_\alpha$ are given by
\beqq
 C_\alpha = \Big(\int_0^\infty x^{-\alpha} \sin x\,dx\Big)^{-1}
=
\begin{cases}
(\Gamma(1-\alpha) \cos (\pi\alpha/2))^{-1}
 &\mbox{if $\alpha\neq 1,$}\\
2/\pi &\mbox{if $\alpha =1,$}
\end{cases}
\eeqq
and
\beqq
b_\alpha = 2^{\alpha/2}\Gamma\left(1+\mbox{ $\frac{\a }{2}$}\right).    \label{1.5}
\eeqq

An important consequence of the representation \eqref{sum:equation}
is that if we condition on the sequences $\{\Gamma_k\}$ and $\{\omega_k\}$
then the conditioned field is stationary Gaussian. This will enable
us, as in the sub-Gaussian case,  to use conditional Gaussian
arguments to prove the following result, in which we establish the
asymptotics of the expected Euler characteristic of the excursion
sets of the real harmonisable stable fields. For simplicity we
restrict ourselves to the case of a compactly supported
control measure $\mu$ with a bounded density, 
but we expect the result remain true in
greater generality. 

\begin{theorem}
\label{theorem-harmonisable}
Let $f$ be a  harmonisable, \SaS,  random field as in
\eqref{integralrep:equation} or \eqref{sum:equation}, defined on the
$N$-rectangle $T$ of \eqref{T:definition}. Assume that the 
control measure $\mu$ has compact support and 
 a bounded density with respect to Lebesgue measure. 
Then
\beq
\label{harmo:main:asymp}
\qquad \E\left\{\p (A_u(f,T))\right\} \asymp   u^{-\a}C_\a\mu_0\left(
\frac{2^{1-\a/2} \Gamma\left(\frac{1+\a}{2}\right)}{\sqrt{\pi} b_\a}  +
\frac{1}{2\pi}\sum_{j=1}^N\mu_j T_j
\right),
\eeq
where $\mu_0$  was defined at \eqref{mu0} and the $\mu_j$, $j=1,\dots,N$ are the
(normalised) moments
\beq
\label{muj}
\mu_j \definedas \E\{ |\omega_j| \}  = \int   |\omega_j| \frac{\mu (d\omega)}{\mu_0}      .
\eeq
If, furthermore, $\mu$ is rotationally invariant (so that $f$ is isotropic) and $M$ is a
compact convex domain, then
\beq
\label{harmo:main-iso:asymp} \\ \notag
  \E\left\{\p (A_u(f,T))\right\} \asymp   u^{-\a}C_\a\mu_0\left(
\frac{2^{1-\a/2} \Gamma\left(\frac{1+\a}{2}\right)}{\sqrt{\pi} b_\a}\lips_0(M)  +
\frac{\mu_1}{2\pi}\lips_1(M)
\right),
\eeq
where $\mu_1$ is any of the (equivalent) moments given by \eqref{muj}.

\end{theorem}

Note, once again, how different this result is to the corresponding Gaussian one, and
even to the sub-Gaussian one. Comparing, for example,
\eqref{harmo:main-iso:asymp} with its Gaussian counterpart
 \eqref{rgeometry:EEC-isotropic-LK:equation} (take  $j=0$ there) we see that while the
leading term in the Gaussian case comes from the volume term $\lips_N(M)$, in the
harmonisable stable case it comes from the two lowest \LKC s, $\lips_0(M)$ and $\lips_1(M)$.

To see why this should be the case, we shall postpone the rather technical
 proof of the theorem for a moment,  take a moment to
describe the principles involved, and then use them to obtain a
 heuristic proof of the theorem.

Consider the representation
\eqref{sum:equation} for harmonisable stable fields. From this,
one can argue that, conditional on $f$ reaching a high
level $u$, the first term, with the coefficient $\Gamma_1^{-1/\alpha}$, will dominate not
only  all the other summands, but in fact their sum. (Formulating this properly, and then
establishing it, is basically the main part of the technical proof.)

This being the case, $f$ will tend, at high levels, to look like a
cosine function, with random height, frequency, and direction. Thus, for example,
in $\real^2$, a typical high level excursion set will look like that in
the square of Figure \ref{stable:figure},  a sequence of
strips looking like the tops of cosine
waves, with small perturbations due to the terms in the representation
 other than the dominant one.

To count how many such strips there are, one needs only to look at the boundary of the
square. In fact, if we increase the size of the square, it is clear that the number of
strips (and so the Euler characteristic) grows proportionately to  the length of the
edges, and not to
the area. Thus it should no longer be surprising that only $\lips_0(M)$ and
$\lips_1(M)$ appear in \eqref{harmo:main-iso:asymp}. (The  $\lips_0(M)$ term arises
to `catch' regions such as those in the lower left corner of  Figure
\ref{stable:figure}.)

\begin{figure}[!ht]
\centerline{\rotatebox{0}{\resizebox{1.5in}{1.5in}{\includegraphics{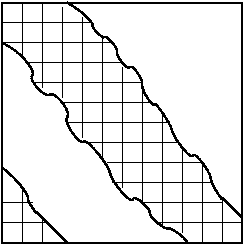}}}}
\caption{A ``typical'' excursion set for a harmonisable field on $[0,1]^2$.}
\label{stable:figure}
\end{figure}

There is also another argument that will give us not only the results
of Theorem \ref{theorem-harmonisable}, but even a little more. Again,
it is only heuristic, 
but it is both elegant and simple, and since it may also be applicable to other
problems it is worth the space we shall devote to it. Its approach is via
integral geometry.

We start with the assumption that it is only the first term of the sum
\eqref{sum:equation} that will be important, and so (ignoring
multiplicative constants) look at the random field 
\begin{equation*}
  f(t) = \Gamma_1^{-1/\alpha}\bigl( G_1^{(1)} \cos(t \omega_1) +
  G_1^{(2)} \sin(t\omega_1)\bigr)
\end{equation*}
over a compact, convex $M$.

For a given direction $\omega \in \real^N$, let $P_{\omega}$ denote
projection onto the line containing the origin and  $\omega$ and let
\begin{equation*}
  M_{\omega} = \{P_{\omega}t\: t \in M\}
\end{equation*}
be the projection of $M$ onto this line.
Conditioning on  $\omega$, the Euler characteristic of the excursion
set $A_u(f,M)$ is the same as the number of upcrossings of the level $u$
by the one-dimensional cosine wave on
$M_{\omega}$, plus one if $f\geq u$ at the boundary point of $M_\omega$ closest to the
origin.

If we now average over the Gaussian variables $G_1^{(1)}$ and   $G_1^{(2)}$,
then it is easy to check that, conditioned on $\omega$, the expected
Euler characteristic, should be proportional to
\begin{equation*}
\P\left\{f(0)\geq u\right\} \ + \
 u^{-\a} \|\omega\| \left|T_{\omega}\right|
\asymp  u^{-\a}\left(k_1 +  k_2  \|\omega\| \left|T_{\omega}\right| \right),
\end{equation*}
for some constants $k_1$ and $k_2$ which we shall not worry about.
Averaging over $\omega$ we find that
\beqq
&&\int_{\real^N}\left( k_1+
k_2 \|\omega\| |T_{\omega} | \right) \frac{\mu(d\omega)}{\mu_0}
= k_1 + k_2 \int_{\real^N}
|T_{\omega}|\|\omega\|  \frac{\mu(d\omega)}{\mu_0}.
\eeqq
Now assume that the measure
$\mu$ is rotationally symmetric. Then the so-called
projection theorem of integral geometry (cf.\
 \cite{Klain:Rota:Geometric:Probability:1997}, \S 7.4) gives us that the integral here
is proportional to $\lips_1(M)$. This gives  \eqref{harmo:main-iso:asymp}, once we compute the constants.

Indeed, we can go further than  \eqref{harmo:main-iso:asymp},
dropping the isotropy assumption. In general, the quantity
$  \|\omega\| |M_{\omega}|$ can be expressed as
\begin{equation}
\label{eq:harmonisable:support:diff}
\|\omega\| |M_{\omega}| = \sup_{t \in M} \langle \omega, t \rangle -
\inf_{t \in M} \langle \omega, t \rangle = h_M(\omega) - h_M(-\omega)
\end{equation}
where
\begin{equation}
\label{eq:harmonisable:support}
  h_M(\omega) = \sup_{t \in M} \langle \omega, t \rangle
\end{equation}
is the support function  \cite{Klain:Rota:Geometric:Probability:1997}
of the convex body $M$. (Usually, the support function is considered as
a function
on the unit sphere $S(\real^N)$, but \eqref{eq:harmonisable:support}
is well defined as  a function on $\real^N$.)

Following through with the constants, this argument would give that,
for compact convex bodies $M$, 
\begin{equation}
\label{harmo:main:asymp:general}
\begin{aligned}
\lefteqn{ \E\left\{\p (A_u(f,M))\right\} \asymp} \\
& \qquad   u^{-\a}C_\a\mu_0\left(
\frac{2^{1-\a/2} \Gamma\left(\frac{1+\a}{2}\right)}{\sqrt{\pi} b_\a} \ +\
\frac{1}{2\pi}\int_{\real^N}( h_M(\omega) - h_M(-\omega)) \frac{\mu(d\omega)}{\mu_0}
\right).
\end{aligned}
\end{equation}

To see how this works for rectangles $T$, note that in this case
the difference of the support functions appearing in
\eqref{eq:harmonisable:support:diff} is just $\sum_{j=1}^N |\omega_j|
T_j$. Substituting this into \eqref{harmo:main:asymp:general} gives us
back \eqref{harmo:main:asymp}.

We shall return to integral geometric arguments later, in Section
\ref{section:isotropic},
 but for the moment we leave heuristics and geometric arguments
and give the promised technical  proof.

\vspace*{0.15truein}
\noindent {\bf Proof of Theorem  \ref{theorem-harmonisable}.}
We shall only give a proof of \eqref{harmo:main:asymp}, the result for
the  rectangular parameter space $T$. The extension to  compact, convex parameter
spaces follows from Hadwiger's representation of additive functionals, as in the
proof of Theorem \ref{theorem-subgaussian}.

The proof will follow the general lines of that of Theorem
\ref{theorem-subgaussian}, in that it begins with a conditioning argument from
which asymptotics are computed. The harmonisable case, however, is somewhat
more complicated.

We begin with the representation  \eqref{sum:equation}.
If we condition on the sequences
$\{\Gamma_k\}$ and   $\{\omega_k\}$, it is immediate that
the conditioned field is a stationary Gaussian field on $\RN$ with
mean zero and covariance function
\beqq
R(t) = \gamma_\alpha^2
\sum_{k=1}^\infty   \Gamma_k^{-2/\alpha}\cos(t\omega_k),
\eeqq
where
\beq
\label{gamma:definition}
\gamma_\alpha \definedas \big(C_\alpha b_\alpha^{-1}
\mu_0\big)^{1/\alpha}.
\eeq
Therefore it has variance
\beq
\label{parameters:sigma:equation}
\tsigma^2 =  \gamma_\alpha^2 \sum_{k=1}^\infty \Gamma_k^{-2/\alpha}
\eeq
and second spectral moments
\beqq
\widetilde\lambda_{ij} =  \gamma_\alpha^2
\sum_{k=1}^\infty \Gamma_k^{-2/\alpha} \omega_k(i)\omega_k(j).
\eeqq

Our next step will be to apply the Gaussian Theorem
\ref{rgeometry:EEC:theorem} to the conditioned harmonisable process, and then use the fact  that
\beqq
\E\left\{\p (A_u(f,T)\right\} =
\E\left\{\E\left\{\p \left(A_u(f,T)\right)\big| \Gamma_k,\ \omega_k,\ k\geq 1\right\}\right\}.
\eeqq
As in the sub-Gaussian case, the iterated expectation requires
justification, which will be provided only later in Lemma
\ref{lemma:harmonisable-finite}.

In order the apply the Gaussian result, we  need first to  verify three conditions:
that the conditioned process is a.s.\ $C^2$, that the joint distributions
of the various conditioned derivatives are non-degenerate, and that the condition
\eqref{theorem-subgaussian} on the moduli of continuity of the conditioned process
is satisfied.

We tackle the first of these first, by showing that $f$, itself, is a.s.\
$C^2$. From this fact and Fubini's theorem, the same will be true of the conditioned
processes.

Note that the assumption of compact support for
the control measure $\mu$ implies, by Corollary 11.7.5 of \cite{ST94},
that  $f$ is absolutely continuous. Write $f_i$ and $f_{ij}$ for the
various first and second order partial derivatives of $f$. Exploiting
the representation \eqref{integralrep:equation}, 
it is easy to check that as  versions of its partial derivatives we
can take the random fields 
\beqq
f_j(t) = {\text{Re} }\Big\{ \int_{\RN} e^{it\omega}\,i\omega(j)\,
Z(d\omega) \Big\}.
\eeqq
Since these are,  again, harmonisable fields with compactly supported
control measures they are all continuous. From this it follows that
$f$ is $C^1$. Applying the same argument to the derivatives of the
$f_j$ gives that $f$ is $C^2$.

We now turn to the issue of the non-degeneracy of joint distribution
of the various derivates. Applying the representation
\eqref{sum:equation}, we have
\beqq
\qquad
f_i(t) = \big(C_\alpha b_\alpha^{-1}
\mu_0\big)^{1/\alpha}
        \sum_{k=1}^\infty \Gamma_k^{-1/\alpha}\omega_k(i)
       \left(-G^{(1)}_k\sin(t\omega_k) +
       G^{(2)}_k\cos(t\omega_k)\right),
\eeqq
and
\beq
\label{fij-rep} \\ \notag
f_{ij}(t)
= -\big(C_\alpha
b_\alpha^{-1} \mu_0\big)^{1/\alpha}
        \sum_{k=1}^\infty \Gamma_k^{-1/\alpha}\omega_k(i)\omega_k(j)
       \left(G^{(1)}_k\sin(t\omega_k) +
       G^{(2)}_k\cos(t\omega_k)\right),
\eeq
$i,j=1,\ldots, N$. In particular, the joint distribution of $f$ and its
       partial derivatives at time zero  (up to a positive multiplicative
       constant) is that of the random vector
\beqq
\left\{ \sum_{k=1}^\infty \Gamma_k^{-1/\alpha} G^{(1)}_k, \
\sum_{k=1}^\infty \Gamma_k^{-1/\alpha} \omega_k(i)\, G^{(2)}_k,
\ \sum_{k=1}^\infty \Gamma_k^{-1/\alpha} \omega_k(i) \omega_k(j)\,
G^{(1)}_k \right\}_{i,j=1}^N.
\eeqq
Since we have assumed that the distribution of the random vectors
$\omega_k$ has a non-vanishing absolutely continuous component, 
this is, with probability 1, a non-degenerate Gaussian vector. 

It remains to check that condition
\eqref{modcont:equn} on the moduli of continuity of the second order
derivatives 
of the  conditioned process is satisfied. By Corollary 11.3.2 of \cite{RFG},
writing $C_{f_{ij}}(t)$ for the covariance function
of the second order partial derivatives $f_{ij}$, this condition will be satisfied if,
 for small enough $t$, and some finite $K,\eta>0$,
\beq
\label{rgeometry:logcovariance2:equation}
\max_{i,j} \left|C_{f_{ij}}(0) - C_{f_{ij}}(t)\right|\
\leq \ K\left|\ln |t|\,\right|^{-(1+\eta)}.
\eeq
However, using the representation \eqref{fij-rep} for the $f_{ij}$ it is easy to compute
an expression for the $C_{f_{ij}}$ and, using the compact support of the $\omega_k(i)$,
to see from this that \eqref{rgeometry:logcovariance2:equation} is not only satisfied,
but that  the bound on the right hand side can be taken of
order $|t|^2$.

With all the conditions checked, we can now apply Theorem  \ref{rgeometry:EEC:theorem}
to the conditioned harmonisable process, from which  it follows that
\beq
\label{bigexpectation:equation}
&&\E\left\{\p (A_u(f,T))\right\} \\ &&\qquad =\E\left\{
e^{-u^2/2\tsigma^2} \sum_{n=1}^{N}
\sum_{J\in \cO_n}
\frac{|J|\,|\widetilde\Lambda_J|^{1/2}}{(2\pi)^{(n+1)/2}\tsigma^{n}}H_{n-1}\left(\frac{u}{\tsigma}\right)
+\Psi\left(\frac{u}{\tsigma}\right)
\right\},   \notag
\eeq
where the expectation in the second line  is over the $\Gamma_k$ and $\omega_k$,
and for $J\in \cal O_n$ the $n\times n$ matrix $\widetilde \Lambda_J$ bears the same
relation to
 the $\widetilde\lambda_{ij}$ that $\Lambda_J$ does to the $\lambda_{ij}$.

The term $\E\left\{\Psi\left({u}/{\tsigma}\right) \right\}$ can be
handled much as in the proof of Theorem \ref{theorem-subgaussian},
using  \eqref{psi:tauberian} to show that
\beqq
\E\left\{\Psi\left(\frac{u}{\tsigma}\right)\right\}
\asymp  u^{-\a}\,
\frac{2^{1-\a/2}C_\a \mu_0 \Gamma\left(\frac{1+\a}{2}\right)}{\sqrt{\pi} b_\a}.
\eeqq
As for the other terms, it is clear, expanding the Hermite polynomials,
that we need to study the asymptotics of terms of the form
\beq
\label{harmonizable:powers}
\E\left\{\frac{|\widetilde\Lambda_J|^{1/2}u^{n-1-2j}}{\tsigma^{2n-1-2j}}
e^{-u^2/2\tsigma^2} \right\},
\eeq
for $0\leq j\leq \lfloor \mbox{$\frac{n-1}{2}$} \rfloor $.

To do this, we need a little notation. For $k\geq 1$ and a  facet $J$
of dimension $n$  define the $n\times n$ matrices $W_k\equiv W_k(J)$
by setting 
\beq\
\label{Wk}
W_k(i,j)=\omega_k(i)\omega_k(j), \qquad i,j\in \sigma(J).
\eeq
Then
\eqref{harmonizable:powers} can be rewritten as
\beq
\label{oneterm:equation} \\  \notag
\gamma_\alpha^{1/2+j}
\E\left\{ \frac{\left| \det\left( \sum_1^\infty \Gamma_k^{-2/\alpha}
      W _k
  \right) \right|^{1/2} u^{n-1-2j}}{ \left(\sum_1^\infty
  \Gamma_k^{-2/\alpha}\right)^{(2n-1-2j)/2}}\exp\left(\frac{-u^2}{2
  \gamma_\alpha^2  \sum_1^\infty \Gamma_k^{-2/\alpha}}\right)
\right\}.
\eeq

Our first step in handling this expectation will be to truncate the sum in the numerator. Note
that, for any two $n\times n$  matrices  $A$ and $B$, the standard
expansion of a determinant shows that
\beqq
\det (A+B) \leq \det (A) + \sum_{m=1}^{n} C_{nm} \|A\|^{n-m} \|B\|^{m},
\eeqq
for some combinatorial constants $C_{nm}$ that we allow to change from
line to line and $\|A\|=\max_{ij}|a_{ij}|$. Furthermore, since
$\sqrt{x+y}\leq \sqrt{x}+\sqrt{y}$, we can apply this to the numerator
of \eqref{oneterm:equation} to see that there is a constant
$C$  such that
\beq
\label{difference1:equation}
&&\Big|\Big(\det\Big( \sum_1^\infty \Gamma_k^{-2/\alpha}
       W_k \Big)\Big)^{1/2} - \Big(\det \Big( \Gamma_1^{-2/\alpha}
       W_1  \Big) \Big)^{1/2}\Big|
\\ && \qquad\qquad
\leq
C\sum_{m=1}^n \Bigl(
       \Gamma_1^{-2/\alpha}\|W_1\|\Big)^{(n-m)/2}
 \Big(\sum_{k=2}^\infty
\Gamma_k^{-2/\alpha}\|W_k\|\Big)^{m/2}.    \notag
\eeq
We now claim that, in view of the above inequality, the expectation in
\eqref{oneterm:equation}
differs from
\beq
\label{oneterm1:equation}  \qquad
\E\left\{ \frac{\left| \det\left(  \Gamma_1^{-2/\alpha} W_1 \right)
  \right|^{1/2} u^{n-1-2j}}{ \left(\sum_1^\infty
  \Gamma_k^{-2/\alpha}\right)^{(2n-1-2j)/2}}
  \exp\left(\frac{-u^2}{2\gamma_\alpha^2 \sum_1^\infty
 \Gamma_k^{-2/\alpha}}\right)   \right\}
\eeq
by no more than a factor of $o(u^{-\a})$.
Note that in the simplest case, when $N=n=1$ and $j=0$, this is precisely
Lemma 2.4 of \cite{AST93}.  To prove the general claim we need here,
observe that, by \eqref{difference1:equation} and the assumption of
the bounded support of the random vectors $\omega_k$,
it is enough to prove that for every $0\leq j\leq \lfloor
\mbox{$\frac{n-1}{2}$} \rfloor $ and $1\leq m\leq n$,
\beq
  \label{e:bound.diff} \\ \notag
u^{n-1-2j}\E\left\{ \frac{ \Gamma_1^{-(n-m)/\alpha}\Big(\sum_{k=2}^\infty
\Gamma_k^{-2/\alpha}\Big)^{m/2}}{\left(\sum_1^\infty
  \Gamma_k^{-2/\alpha}\right)^{(2n-1-2j)/2}}
  \exp\left(\frac{-u^2}{2\gamma_\alpha^2 \sum_1^\infty
  \Gamma_k^{-2/\alpha}}\right)   \right\}
= o(u^{-\a}),
\eeq
as $u\to\infty$. To this end, note that the left hand side of
\eqref{e:bound.diff} can be bounded by
\begin{equation} \label{e:bound.diff1}
u^{n-1-2j}\E\left\{ \frac{ \Big(\sum_{k=2}^\infty
\Gamma_k^{-2/\alpha}\Big)^{1/2}}{\left(\sum_1^\infty
  \Gamma_k^{-2/\alpha}\right)^{(n-2j)/2}}
  \exp\left(\frac{-u^2}{2\gamma_\alpha^2 \sum_1^\infty
  \Gamma_k^{-2/\alpha}}\right)   \right\} \,.
\end{equation}
Take a small $\vep>0$ and write
\eqref{e:bound.diff1} as a sum of two terms, the first when
the expectation is restricted  the event $\bigl\{
\bigl(\sum_{k=2}^\infty \Gamma_k^{-2/\alpha}\big)^{1/2}>\vep
u\bigr\}$, and the second arising  when
the expectation is restricted  the complementary
 event.

Noting that for any $\theta>0$ there is a finite
$c$ such that $\sigma^{-2\theta}\exp\{ -u^2/\sigma^2\}\leq
cu^{-2\theta}$ for all $u,\sigma>0$, we see that the first term can be
bounded by
\beqq
c\, u ^{-1}\E\left\{ \Big(\sum_{k=2}^\infty
\Gamma_k^{-2/\alpha}\Big)^{1/2}
\indic_{\left(\sum_{k=2}^\infty
\Gamma_k^{-2/\alpha}\right)^{1/2}>\vep u}  \right\}.
\eeqq
Since $P\bigl( \sum_{k=2}^\infty \Gamma_k^{-2/\alpha}>u\bigr)= o\bigl(
u^{-\alpha/2}\bigr)$ (see \cite{ST94}), it follows that the first term
in \eqref{e:bound.diff1} is $o\bigl( u^{-\alpha})$ as $u\to\infty$ for
every fixed $\vep>0$. On
the other hand, the second term  is, clearly,
bounded by
$$
\vep u^{n-2j}E\left\{ \frac{1}{\left(\sum_1^\infty
  \Gamma_k^{-2/\alpha}\right)^{(n-2j)/2}}
  \exp\left(\frac{-u^2}{2\gamma_\alpha^2 \sum_1^\infty
  \Gamma_k^{-2/\alpha}}\right)   \right\}\,.
$$
Since the sum $\sum_1^\infty   \Gamma_k^{-2/\alpha}$ is a positive
$\alpha/2$-random variable, it has a density of asymptotic
order $x^{-(1+\alpha/2)}$, and straightforward  estimates now show that the expectation
above is of order $u^{-(n-2j+\alpha)}$ as $u\to\infty$.
Consequently,
\beqq
\limsup_{u\to\infty} u^\alpha\ \bigl( \text{the second term in
  \eqref{e:bound.diff1}}\bigr) \leq c\vep
\eeqq
for some $c>0$, and letting $\vep\to 0$ proves \eqref{e:bound.diff}.

Returning now to \eqref{oneterm1:equation}, which is what is left to study, we note
that it has a simple structure, since it is
immediate from the definition of  $W_1=W_1(J)$ that it is a rank one matrix for all $J\in \cal O_n$.
 Hence,
\beqq
\det\left(  \Gamma_1^{-2/\alpha} W_1  \right)
=\Gamma_1^{-n/\alpha} \det \left(W_1\right)\equiv 0,
\eeqq
unless $n=1$. Thus, taking $n=1$ and $j=0$ (now the only possible
value of $j$) in \eqref{oneterm1:equation}, recalling the
independence of the $\omega_k$ and $\Gamma_k$, we need only consider
$N$ terms of the form
\beqq
\E\left\{ \left|\omega_1(i)\right|\right\}
 \E\left\{ \frac{   \Gamma_1^{-1/\alpha}}{ \left(\sum_1^\infty
  \Gamma_k^{-2/\alpha}\right)^{1/2}}\exp\left(\frac{-u^2}{2 \gamma_\alpha^2
  \sum_1^\infty \Gamma_k^{-2/\alpha}}\right)       \right\}, \
 i=1,\ldots, N.
\eeqq
The first expectation here is, by definition, $\mu_i$, while the the second
converges to the constant $\mu_0C_\alpha$ by the first part of Theorem
 2.2 of  \cite{AST93}. Putting everything together proves
 \eqref{harmo:main:asymp}. \qed
\vspace*{0.15truein}

\section{Concatenated-harmonisable fields}
\label{main:wavelimited}
For our final class of examples we shall introduce a class of random fields
which, to the best of our knowledge, have not been studied earlier. We actually discovered
them by looking for a class of examples which `interpolated' between the sub-Gaussian
ones, for which all the $\lips_j$ appear in the asymptotic formula for the mean
Euler characteristic of excursion sets, and the harmonisable ones, for which only
$\lips_0$ and $\lips_1$ appear. However, having found them for this rather artificial
purpose, we believe that they actually present an interesting class of stable fields that
will provide useful models in applied settings.

 To define this new class of random fields, we take the representation
\beq
\label{wl-sum:equation}  \qquad
 f(t) = \left(\frac{C_\alpha\mu_0}{ b_\alpha}\right)^{1/\alpha}
        \sum_{k=1}^\infty \Gamma_k^{-1/\alpha}
        \sum_{\ell=1}^{N'}
    \left(G^{(1)}_{k\ell}\cos(t\omega_{k\ell}) + G^{(2)}_{k\ell}\sin(t\omega_{k\ell})\right).
\eeq
where the $\Gamma_k$ are as in \eqref{sum:equation}, the
 $\{G^{(i)}_{k\ell}\}$, $i=1,2$, $\ell=1,\dots,N'$, are
independent, standard Gaussian random variables, and the
 $\{\omega_{k\ell}\}$ are independent with the distribution of the $\omega_k$
of \eqref{sum:equation}. The parameter $N'$ satisfies $1\leq N'\leq N$.
(When $N'=1$ we  recover the harmonisable fields of \eqref{sum:equation}.)
We call such fields concatenated-harmonisable, the adjective ``concatenated'' coming
 from the innermost sum in \eqref{wl-sum:equation}.

If the heuristics used before work again the dominant term
in the expansion,
\beqq
\left(\frac{C_\alpha\mu_0}{ b_\alpha}\right)^{1/\alpha}
\Gamma_1^{-1/\alpha}  \sum_{\ell=1}^{N'}
 \left(G^{(1)}_{k\ell}\cos(t\omega_{k\ell}) +
 G^{(2)}_{k\ell}\sin(t\omega_{k\ell})\right), 
\eeqq
should determine the properties of the high level excursion sets.
This random field is quite different from that of the simple random
wave generated by the first term  of the harmonisable processes, and
so the arguments there, based on examples as in Figure
\ref{stable:figure}, are not going to carry over easily to the current
situation. We did find an integral geometric argument which justified
Theorem \ref{theorem-concat} below, but it was no longer simple and,
since it was also non-rigorous, we shall not bring it here.

In order to state the result, we need a little more notation.
Changing slightly that of the proof of Theorem \ref{theorem-harmonisable},
 choose a facet $J\in \cal O_n$, and, for each
$k\geq 1$, define
the $n\times n$ matrix $W_k(J)$ with elements
\begin{equation}
\label{wkJ:defn1}
\left(W_k(J)\right)_{ij} =  \sum_{\ell=1}^{N'}
\omega_{k\ell}(i)\omega_{k\ell}(j), 
\end{equation}
for  $i,j\in \sigma(J)$.
Furthermore, define the ($k$ independent) parameters
\beq
\label{LambdaJ-concat}
\Lambda (J)\definedas \E\left\{ \left| \det \left(W_k (J)\right)\right|^{1/2} \right\}.
\eeq
%

\begin{theorem}
\label{theorem-concat}
Let $f$ be a  concatenated-harmonisable, \SaS\
random field as in  \eqref{wl-sum:equation},
defined on the $N$-rectangle $T$ of \eqref{T:definition}.
Assume that the control measure $\mu$
has  compact support, and a bounded density with respect to the
Lebesgue measure. Then
\beq
\label{concatharmo:main:asymp}
\qquad \E\left\{\p (A_u(f,T))\right\} &\asymp&   u^{-\a}\mu_0 C_\a\Bigg(K_0
 +
 \sum_{n=1}^{N'}  K_n \sum_{J\in\cal O_n}
 |J| \Lambda (J)
\Bigg),
\eeq
where
\beqq
K_0 = \frac{2^{1-\a/2}
  \Gamma\left(\frac{1+\a}{2}\right)(N')^{\a/2}}{\sqrt{\pi} b_\a}, 
\eeqq
and, for $n=1,\dots,N'$, $K_n$ is given by \eqref{Kn} below. 

If, furthermore, $\mu$ is rotationally invariant (so that f is
isotropic) and $M$ is a compact convex domain, then
\beq
\label{concatharmo:main-iso:asymp}
\qquad \E\left\{\p (A_u(f,M))\right\} &\asymp&   u^{-\a}\mu_0 C_\a
 \sum_{n=0}^{N'}  K_n \Lambda_n \lips_n(M),
\eeq
where $\Lambda_n$ is given by \eqref{LambdaJ-concat} for any
 $n$-dimensional facet $J$. 

\end{theorem}

\begin{proof} We shall prove only \eqref{concatharmo:main:asymp}, with the result
\eqref{concatharmo:main-iso:asymp} for isotropic processes on convex parameter sets
following from the usual integral geometric argument via Hadwiger's theorem.

The proof follows the lines of that of Theorem \ref{theorem-harmonisable}, and so
we start by conditioning on the sequences
$\{\Gamma_k\}$ and   $\{\omega_{k\ell}\}$,  to obtain a stationary Gaussian process on
$\RN$ with mean zero, variance
\beqq
\tsigma^2 =  \gamma_\alpha^2 N' \sum_{k=1}^\infty \Gamma_k^{-2/\alpha},
\eeqq
where $\gamma_\alpha$ is as at \eqref{gamma:definition},
and with second order spectral moments
\beqq
\widetilde\lambda_{ij} &=&  \gamma_\alpha^2
\sum_{k=1}^\infty \Gamma_k^{-2/\alpha}
\sum_{\ell=1}^{N'} \omega_{k\ell}(i)\omega_{k\ell}(j).
\eeqq

We need to prove that the iterated expectation argument for computing
$\E\left\{\p (A_u(f,T))\right\}$ is valid. This is done in Lemma
\ref{lemma:harmonisable-finite} below. The arguments that worked in
the harmonisable case also work here to show that the 
conditionally Gaussian process satisfies all the conditions of Theorem
\ref{rgeometry:EEC:theorem}.

Thus, our task becomes one of evaluating \eqref{bigexpectation:equation} once
again, albeit with the new definitions of the variables there.
The term involving $\Psi$ changes from that in the previous proof
only insofar as there is now an additional factor of
$N'$ in the definition of $\tsigma^2$ in \eqref{parameters:sigma:equation}
and this gives the first term in \eqref{concatharmo:main:asymp}
(cf.\ Lemma \ref{l:tauberian}.)

As far as the other terms are concerned, the argument is identical to
that in the proof of Theorem  \ref{theorem-harmonisable} as far as
\eqref{oneterm1:equation}, and so what remains to compute is
\beqq
(N^\prime)^{-(2n-1-2j)/2}\E\left\{ \frac{\left| \det\left(
  \Gamma_1^{-2/\alpha} W_1 \right) \right|^{1/2} u^{n-1-2j}}{
  \left(\sum_1^\infty \Gamma_k^{-2/\alpha}\right)^{(2n-1-2j)/2}}
\exp\left(\frac{-u^2}{2\gamma_\alpha^2 N^\prime 
  \sum_1^\infty \Gamma_k^{-2/\alpha}}\right) \right\}
\eeqq
for $1\leq n\leq N'$ and $0\leq j\leq \lfloor \mbox{$\frac{n-1}{2}$}
  \rfloor $. The restrictions on $n$ follow from the fact that the
  matrices  $W_k(J)$ have rank not exceeding $N^\prime$ and, since
  the $\omega_{k\ell}$ have a distribution with a non-vanishing
  absolutely continuous component, this rank will equal $\min (n,N')$,
  with  positive probability. 

Since $W_1$ is independent of the $\Gamma_k$, we can rewrite the above as
\beqq
&&(N^\prime)^{-(2n-1-2j)/2}\E\left\{|\det (W_1(J))|^{1/2}\right\}
\\ && \qquad\qquad\qquad\times
\E\left\{ \frac{  \Gamma_1^{-n/\alpha}  u^{n-1-2j}}{
  \left(\sum_1^\infty
  \Gamma_k^{-2/\alpha}\right)^{(2n-1-2j)/2}}
\exp\left(\frac{-u^2}{2\gamma_\alpha^2 N^\prime 
  \sum_1^\infty \Gamma_k^{-2/\alpha}}\right)       \right\}.
\eeqq

The first expectation here is, by definition, $\Lambda (J)$. As for
the second, its asymptotics follow from Lemma \ref{lemma:last} below,
to give 
%
\beqq
u^{n-1-2j}\times \mu_0 C_\a C_{nj}  u^{-(\alpha +(n-1-2j))}
=  \mu_0 C_\a C_{nj}  u^{-\alpha},
\eeqq
where
\beqq
C_{nj}\definedas
\a C_{\a/2}\sigma_\a^{\a/2}b_\a^{-1}2^{(\a+n-2j-3)/2}
\Gamma\left(\text{$\frac{\a+n-1-2j}{2}$}\right)
\gamma_\a^{n-1-2j}(N^\prime)^{(\alpha-n)/2}. 
\eeqq
If we now substitute this back into \eqref{bigexpectation:equation},
collect all the constants appearing in the Hermite polynomials, and
define 
\beq
\label{Kn}
K_n=\frac{(n-1)!}{(2\pi)^{(n+1)/2}}\sum_{j=0}^{ \lfloor
  \mbox{$\frac{n-1}{2}$} \rfloor}\frac{(-1)^j C_{nj}}{j!(n-1-2j)!2^j}, 
\eeq
a few lines of algebra yield \eqref{concatharmo:main:asymp}, 
and we are done.
\qed
\end{proof}

\begin{lemma}
\label{lemma:last}
Maintaining the above notation, set
\beqq
X=\sum_{j=1}^\infty \Gamma_j^{-2/\alpha},\qquad n\geq 0,\  \beta>n/2,
\eeqq
and take $\gamma>0$. Then
\beq
\label{bound:last}
&&{\E\left\{ \Gamma_1^{-n/\alpha} X^{-\beta}e^{-u^2/2\gamma^2X}\right\} }
\\ \notag &&\quad \asymp
\alpha 2^{(\a+2\beta -n -2)/2} C_{\a/2}\sigma_\a^{\a/2}
\gamma^{\a+2\beta-n} \Gamma(\beta +\smallhalf(\alpha-n)) u^{-(\alpha+2\beta-n)}.
\eeq
\end{lemma}

\begin{proof}
The proof will proceed by establishing asymptotic upper and lower
bounds for the expectation, which we shall denote by $I_u$.

For the upper bound,  note that since $\Gamma_1^{-2/\alpha}\leq X$ it
is immediate that
\beq
\label{equn:last}
I_u \leq \E\left\{  X^{-(\beta- n/2)}e^{-u^2/2\gamma^2 X}\right\}.
\eeq
However, the asymptotics of the right hand side are covered by Lemma
\ref{l:tauberian}, and so we have that the right hand side of
\eqref{bound:last} provides an upper bound for the asymptotics of $I_u$.

For the lower bound, fix $\e >0$ and again exploit the fact that
$\Gamma_1^{-2/\alpha}\leq X$ to see that
\beqq
I_u &\geq& \E\left\{\Gamma_1^{-n/\alpha}
X^{-\beta}e^{-u^2/2\gamma^2\Gamma_1^{-2/\alpha}} \right\}\\
&\geq& (1+\epsilon)^{-\beta}
\E\left\{ \Gamma_1^{-n/\alpha}
(\Gamma_1^{-2/\alpha})^{-\beta}e^{-u^2/2\gamma^2\Gamma_1^{-2/\alpha}}
\indic_{\sum_{j=2}^\infty \Gamma_j^{-2/\alpha}\leq \e\Gamma_1^{-2/\alpha}}\right\}
\\ &\geq& (1+\epsilon)^{-\beta}
\E\left\{ \Gamma_1^{-n/\alpha}
(\Gamma_1^{-2/\alpha})^{-\beta}e^{-u^2/2\gamma^2\Gamma_1^{-2/\alpha}}
\indic_{Y \leq \e \Gamma_1^{-2/\alpha}}\right\},
\eeqq
where $Y$ is a copy of $X$ independent of $\Gamma_1$. Now fix
$M>0$ and note
\beqq
I_u &\geq& (1+\e )^{-\beta}\P\{X\leq M\} \E\left\{ \Gamma_1^{-n/\alpha}
(\Gamma_1^{-2/\alpha})^{-\beta}e^{-u^2/2\gamma^2\Gamma_1^{-2/\alpha}}
\indic_{\Gamma_1^{-2/\alpha}\geq M/\e}\right\}\\
&=& (1+\e )^{-\beta}\P\{X\leq M\} \E\left\{
(\Gamma_1^{-2/\alpha})^{-(\beta-n/2)}e^{-u^2/2\gamma^2\Gamma_1^{-2/\alpha}}
\indic_{\Gamma_1^{-2/\alpha}\geq M/\e}\right\}
.
\eeqq
The expectation here is similar to that in \eqref{equn:last} and Lemma  \ref{l:tauberian},
with $\Gamma_1^{-2/\alpha}$ replacing $X$. However, these two random variables have
precisely the same tail behavior, and a check of the proof of Lemma  \ref{l:tauberian}
(cf.\ Lemma 2.2 of \cite{AdlerSam97}) shows that this is all that entered into the
asymptotic behavior of the expectation. Hence
\beqq
I_u \stackrel{>}{\sim}
(1+\e )^{-\beta}\P\{X\leq M\}
 \E\left\{  X^{-(\beta+ n/2)}e^{-u^2/2\gamma^2 X}\right\}.
\eeqq
Sending  $\e\to 0$ and $M\to\infty$ completes the proof.
\qed
\end{proof}

\section{On the mean \LKC s of excursion sets}
\label{section:isotropic}
Throughout this paper, we have concentrated on the expected Euler characteristics of
excursion sets. However, at least in the isotropic cases, these results are immediately
extendable to expected \LKC s, via a result known as Crofton's formula.

To state Crofton's formula we start with
the affine Grassmanian $\graff(N,k)$ of all $k$-dimensional flats in $\real^N$;
{\it viz.}\ of all $k$-dimensional linear subspaces of $\real^N$  not necessarily
passing through the origin. On $\graff(N,k)$ there is a natural Haar measure, known
as kinematic measure, which we denote by $\lambda_k^N$. Its precise definition and
normalisation will not be important to us.  Crofton's formula states that
\beq
\label{crofton:formula}
\int_{\graff(N,N-k)} \lips_j(M \cap V) \; d\lambda^N_{N-k}(V) =
\sqbinom{k+j}{j} \lips_{k+j}(M).
\eeq

Now let $f$ be one of the stable random fields of this paper, and $M$ a compact, convex set
in $\RN$. If $f$ is isotropic, then, recalling that the Euler functional $\p$ is
also $\lips_0$, all of our results can be written in the form
\beq
\label{generic}
\E \left\{  \lips_0 (A_u(f,M))\right\}
= \sum_{k=0}^N C_{k}(u) \lips_k(M),
\eeq
for some functions $ C_{k}(u)$. The $ C_{k}(u)$ depend on the parameters of $f$ and,
quite often, are identically zero. What is important, however, is that they are
dependent neither  on $M$ nor on $N=\dim (M)$.

Since if $M$ is compact and convex so is $M\cap V$, for any $V\in \graff(N,k)$.
Thus, exploiting Crofton's formula twice, and \eqref{generic} once, we have
\beq
\E \left\{  \lips_j(A_u(f,M))\right\}
 &=& \E \left\{ \int_{\graff(N,N-j)}
\lips_0\left(A_u(f,M) \cap V\right)
\, d\lambda^N_{N-j}(V) \right \}  \notag \\  \label{croftontrick}
&=&
\int_{\graff(N,N-j)}
\E \left\{\lips_0\left(A_u(f,M) \cap V\right) \right \}
\, d\lambda^N_{N-j}(V) \\
&=& \int_{\graff(N,N-j)}
\sum_{k=0}^N C_{k}(u) \lips_k(M\cap V)
 \, d\lambda^N_{N-j}(V) \notag \\
&=& \sum_{k=0}^{N} C_{k}(u) \int_{\graff(N,N-j)}
\lips_k(M \cap V)  \, d\lambda^N_{N-j}(V) \notag \\
&=& \sum_{k=0}^{N-j} \sqbinom{j+k}{k} C_{k}(u)  \lips_{j+k}(M), \notag
\eeq
the change on the range of summation coming from the fact that $\lips_l(M)\equiv 0$ for all
$l>N$. 

Using this, all the formulae in this paper for the asymptotics of mean Euler characteristics,
in the isotropic cases, can be extended to asymptotics for mean \LKC s.

Of course, something has to be said about justifying the change of order of
integration and expectation at \eqref{croftontrick}, which we shall do in
Lemma \ref{lemma:croftontrick} below.

\section{Appendix: On the finiteness of expectations}
\label{section:appendix2}
A  serious technical point which we avoided throughout the paper
was justifying the conditional expectation arguments
\beq
\label{appendix:expectations}
\E\left\{\p \left(A_u(f,T)\right)\right\}=
\E\left\{\E\left\{\p \left(A_u(f,T)\right)\,\big|\,\cal F\right\}\right\}
\eeq
where $f$ was stable, and $\cal F$ was the information we
 assumed to make $f$ conditionally Gaussian.

The need for justification lies in the fact that \EC s need not be
positive. (Were they positive, \eqref{appendix:expectations} would always
hold, with both sides being finite or infinite together.) We shall show that
\beq
\label{appendix:absolute:expectations}
\E\left\{|\p \left(A_u(f,T)\right)|\right\} < \infty,
\eeq
which is sufficient for \eqref{appendix:expectations}, 
restricting ourselves to rectangular parameter spaces only. The same
arguments  can also be applied for general compact domains, the only
difference being a heavier investment in notation.

Before starting the proof of \eqref{appendix:absolute:expectations},
 we note that all the random fields that we considered in
this paper are suitably regular Morse functions, in the terminology of
Chapter 6 of \cite{RFG}. This follows from a Fubini argument and
the fact that, with probability one the conditionally
Gaussian random fields on which we based all our calculations are of 
this kind. This, in turn,  follows from the fact that they all satisfy
the conditions of Theorem \ref{rgeometry:EEC:theorem}, which, by 
Corollary 11.3.2 of \cite{RFG}, implies that they are suitably regular Morse
functions.

Turning now to the proof of  \eqref{appendix:absolute:expectations},
let $J$, as usual, be a $n$-dimensional facet of a $N$-dimensional rectangle $T$. Then, for a Morse function,
 the \EC\ of the excursion set $A_u(f,T)$
can be represented as an alternating sum, over all $J$,  of  the
numbers of critical points of various indices
of $f_{|J}$ over  $A_u(f_{|J},J)$ (cf.\ Section 9.4 of \cite{RFG}).
This sum is trivially no greater, in absolute value, than the
total number, over all $J$, of the number of critical points of $f_{|J}$ over  $J$.
Thus, in order to establish \eqref{appendix:absolute:expectations}, it
will suffice to show that for every facet $J$ 
\beqq
\E\left\{\E\left\{ N(f,J)\,\big|\,\cal F\right\}\right\} < \infty,
\eeqq
where $N(f,J)$ is the number of critical points of $f$ when restricted to $J$.

In order to compute the inner expectation, recall $f$ is conditionally a stationary Gaussian
random field in all the cases of interest to us. Call this process $\hf$. Then
it follows immediately from
Theorem 11.2.1 of \cite{RFG} and the independence of first and second order partial
derivatives of Gaussian fields (cf.\ \cite{RFG} Section 11.7)  that
\beq
\label{appendix:gaussian:expectations}
\E\left\{ N(\hf,J)\, \big| \, \cal F\right\} =
\frac{\E\left\{\big|\det \nabla^2\hf_{|J}\big| \right\}}{\det \Sigma_{\hf_{|J}}},
\eeq
where $ \nabla^2\hf_{|J}$ is the $n\times n$ matrix of second order derivatives of $\hf_{|J}$ and
$\Sigma_{\hf_{|J}}$ the $n\times n$
covariance matrix of its first order derivatives.

With these preliminaries behind us, we can begin establishing the two main results of this
appendix, Lemmas \ref{appendix:subgauss} and \ref{lemma:harmonisable-finite}.

\begin{lemma}
\label{appendix:subgauss}
Let $f$ be a sub-Gaussian process satisfying the conditions of
Theorem \ref{theorem-subgaussian}. Then, for every $u\in \real$,
$\E\{|\p (A_u(f,T))|\}$ is finite.
\end{lemma}
\begin{proof}
Recall that in the sub-Gaussian case $f$ is no more that $X^{1/2}g$,
where $g$ is Gaussian and 
$X$ is positive $\alpha/2$-stable. Thus,  given $X=x$,
 the conditioned process $\hf$ of \eqref{appendix:gaussian:expectations}
is no more that $x^{1/2}g$. It therefore follows that for a facet $J$
of dimension $n$, 
\beqq
\E\left\{ N(f,J)\right\} &=& \E\left\{\E\left\{ N(\hf,J)\big|
X\right\}\right\}\\ 
&=& \E\left\{\left.
\frac{\E\left\{\left|\det \nabla^2\hf_{|J}\right| \right\}}{\det
  \Sigma_{\hf_{|J}}} \ \right|\ X \right\}\\
&=&
\E\left\{\left.
\frac{ X^{n/2} \E\left\{\left|\det \nabla^2g_{|J}\right| \right\}}{X^{n/2}\det \Sigma_{g_{|J}}}
\ \right|\  X \right\}
\\ &=&
\frac{\E\left\{\left|\det \nabla^2g_{|J}\right| \right\}}{\det \Sigma_{g_{|J}}}.
\eeqq
The last expression is purely Gaussian, and under the assumed
conditions of non-degeneracy, also clearly finite, so we are done.
\qed

\end{proof}

Before turning to concatenated-harmonisable processes we require a
technical lemma.

\begin{lemma}
\label{technical}
 Let $A$ and $B$ be two $n\times n$ matrices, and
suppose that $A$ is of rank $m\leq n$. Then each term in the standard Laplace
 expansion of  $\det(A+B)$ involves at most $m$ elements of $A$.
\end{lemma}
\begin{proof}
We shall prove the result for $m=1$. The result for general $m$ then follows
by writing a matrix of rank $m$ as the sum of $m$ matrices of rank one.

 Since $A$ has rank one, there is a vector $a$ and numbers $\theta_1,\ldots,
\theta_n$ such that the $j$-th row of $A$ is $\theta_ja$,
$j=1,\ldots, n$. Let $b_j$ be the $j$-th row of  $B$.
 Writing the determinant of $A+B$ as a function of its rows  we have
\beqq
\det(A+B) &=& \det\bigl( \theta_1a+b_1, \theta_2a+b_2,\ldots,
\theta_na+b_n\bigr)
\\
&=& \theta_1 \det\bigl( a, \theta_2a+b_2,\ldots,
\theta_na+b_n\bigr)\\ &&\qquad\qquad + \det\bigl( b_1, \theta_2a+b_2,\ldots,
\theta_na+b_n\bigr).
\eeqq
Note that
\beqq
&&\theta_1\det\bigl( a, \theta_2a+b_2,\ldots,
\theta_na+b_n\bigr) \\
&&\qquad\qquad= \theta_1\theta_2 \det\bigl( a, a,\ldots,
\theta_na+b_n\bigr) + \theta_1\det\bigl( a, b_2,\ldots,
\theta_na+b_n\bigr)\\
&&\qquad\qquad=  \theta_1\det\bigl( a, b_2,\ldots, \theta_na+b_n\bigr) \\
 &&\qquad\qquad= \ldots = \theta_1 \det\bigl( a, b_2,b_3\ldots, b_n\bigr).
\eeqq
Similarly,
\beqq
\det\bigl( b_1, \theta_2a+b_2,\ldots,
\theta_na+b_n\bigr) = \sum_{k=2}^n \theta_k \det\bigl( b_1,\ldots,
b_{k-1}, a, b_{k+1},\ldots, b_n\bigr),
\eeqq
and continuing this process leads to
\beqq
\det(A+B) = \det B + \sum_{k=1}^n \theta_k \det\bigl( b_1,\ldots,
b_{k-1}, a, b_{k+1},\ldots, b_n\bigr),
\eeqq
from which the result follows.
\qed
\end{proof}

An immediate consequence of Lemma  \ref{technical}  is
\begin{corollary}
\label{appendix:corollary}
Let $A_1,\dots,A_n$ be $n\times n$
matrices of rank $m\leq n$. Then, for any  real $r_1,\ldots, r_n$,
\beqq
\det\Bigl( \sum_{j=1}^n r_jA_j\Bigr) = r_1^m\ldots r_n^m\, \det\Bigl(
\sum_{j=1}^n A_j\Bigr).
\eeqq
\end{corollary}

We now have what we need to prove our last lemma.

\begin{lemma}
\label{lemma:harmonisable-finite}
Let $f$ be a harmonisable, or concatenated-harmonisable,  
\sas\ random field satisfying
 the conditions of  Theorem \ref{theorem-harmonisable},
 or Theorem \ref{theorem-concat}, respectively. Then
$\E\{|\p (A_u(f,T))|\}<\infty$,  for every $u\in \real$.
\end{lemma}
\begin{proof}
We shall prove the result in the notation of the  concatenated-harmonisable
case, since taking $N'=N$ gives the harmonisable case.
In view of the argument leading to
\eqref{appendix:gaussian:expectations}, and the structure of the
conditional process $\hf$ in this case (cf.\ the proof of Theorem
\ref{theorem-harmonisable} for this and following notation)
 we need to show that
\begin{equation} \label{e:0}
\E\left\{ \frac{\E\left\{\left|\det \sum_{k=1}^\infty
  \Gamma_k^{-1/\alpha}G_k 
W_k\right|\right\}}{\left(\det \sum_{k=1}^\infty
  \Gamma_k^{-2/\alpha}W_k\right)^{1/2}}\right\} <\infty,
\end{equation}
where the inner expectation is taken only with respect to the Gaussian
random variables. Recall also that each $W_k$, given by
\eqref{wkJ:defn1}, can be written in the form 
\begin{equation} \label{e:split.W}
W_k = \sum_{l=1}^{N^\prime}W_{k,l}\,,
\end{equation}
where $\bigl( W_{k,l},\, k\geq 1, \, l=1,\ldots, N^\prime\bigr)$ are
i.i.d., rank 1, $n\times n$,  random matrices of the form \eqref{Wk}. 

We concentrate on the denominator in \eqref{e:0} first. Set
\beqq
d = \left\lfloor \frac{n}{N^\prime}\right\rfloor,
\eeqq
and recall that 
the determinant of the sum of non-negative definite
matrices is at least as large as the determinant of each  of the
terms, so that
\beqq
&&\det \left(\sum_{k=1}^\infty   \Gamma_k^{-2/\alpha}W_k\right)
\\ &&\quad \geq  \det \left(\sum_{k=1}^{d}
\Gamma_k^{-2/\alpha}W_k  + \Gamma_{d+1}^{-2/\alpha} 
\Bigl( W_{d+1,1}+\ldots W_{d+1,n-dN^\prime}\Bigr)
\right)
\\ &&\quad =  \left(\prod_{k=1}^{d}
\Gamma_k^{-2/\alpha}\Gamma_{d+1}^{-2/\alpha} \right)^{N^\prime}
\, \det \left(\sum_{k=1}^{d}W_k  
+\Bigl( W_{d+1,1}+\ldots W_{d+1,n-dN^\prime}\Bigr)
\right),
\eeqq
the last line following from Corollary
\ref{appendix:corollary}. Set 
\beqq
D = \det \left(\sum_{k=1}^{d}W_k 
+ \Bigl( W_{d+1,1}+\ldots W_{d+1,n-dN^\prime}\Bigr)
\right),
\eeqq
so that, to prove
\eqref{e:0}, we need only check that 
\beq
\label{remains}
\E\left\{ \frac{\E\left\{\left|\det \sum_{k=1}^\infty
  \Gamma_k^{-1/\alpha}G_k 
W_k\right|\right\}}{D\left(\prod_{k=1}^{d}
\Gamma_k^{-2/\alpha}\Gamma_{d+1}^{-2/\alpha}
\right)^{N^\prime}}\right\}<\infty.
\eeq
For later reference, note that $D$ is a polynomial function of $n$
i.i.d.\ random vectors, each of length $n$. Under the conditions of the 
theorem, the $n^2$ random variables defining $D$ have a bounded joint density 
with compact support in $\real^{n^2}$. We claim that it follows from this
that there exists a finite $C$ such that, for $\e$ small enough, 
\beq
\label{Dsmall}
\P\{D\leq \e \} \leq C\e.
\eeq
In turn, from this it trivially follows that 
\beq
\label{EDfinite}
\E\{D^{-1/2}\}<\infty.
\eeq
To prove \eqref{Dsmall}, first write $X_1,\dots,X_n$ for the $n$
random vectors defining the determinant $D$, and let 
$A\subset \real^n$ denote the support of each one.
Fix $\e >0$ and define
\beqq
A_\e &=&  \left\{x=(x_1,\dots,x_n)\in A^n\: 
\min_{1\leq i,j\leq n} |x_i(j)|\leq \e \right\}, \\
B_\e &=&  \left\{x=(x_1,\dots,x_n)\in A^n\: D(x_1,\dots,x_n)\leq \e\right\}.
\eeqq
Then
\beq
\label{musum}
\P\{D\leq \e \} \leq \nu\left(A_\e\right) + 
\nu\left(B_\e\cap (A^n\setminus A_\e)\right),
\eeq
where $\nu$ is the measure on $A^n$ generated by the random variables
$X_1,\dots,X_n$. Both of these terms are easily seen to be bounded by a 
constant multiple of $\e$, of $\e$ is small enough. The first is small, since
$A_\e$ can be covered by a finite
 number of $n$-dimensional rectangles, one of 
whose sides has length no more than $2\e$ and the remaining sides having
length no more than the diameter of $A$. Since $\nu$ has bounded
density, this gives that $\mu\left(A_\e\right) \leq C\e$ for some finite $C$.

As far as the second term in \eqref{musum} is concerned, we note that,
because of the smoothness of the mapping $D\:\real^{n^2}\to\real$ in the
region $A^n\setminus A_\e$, the set $B_\e\cap (A^n\setminus A_\e)$ is 
a $C^\infty$, locally convex, stratified manifold, and so it follows from the
generalized tube formula of \cite{RFG} (Theorem 10.9.5) that this 
set also has $\nu$-measure bounded by $C\e$.

Returning now to \eqref{remains}, note that by expanding the determinant 
and applying  Lemma \ref{technical} 
we have that 
\eqref{remains}  is equivalent to
\beqq
\E\left\{ \frac{\E\left\{\left|
\sum_{k_1=1}^\infty 
\ldots \sum_{k_n=1}^\infty
\Gamma_{k_1}^{-1/\alpha}
 \ldots \Gamma_{k_n}^{-1/\alpha}
G_{k_1} \ldots G_{k_n}\, B_{k_1 \ldots k_n}\indic_F
 \right|\right\}}{D\left(\prod_{k=1}^{d}
\Gamma_k^{-2/\alpha}\Gamma_{d+1}^{-2/\alpha}
\right)^{N^\prime}}\right\},
\eeqq
where the
$ B_{k_1\ldots k_n}$ are uniformly bounded random
variables,  independent of $\Gamma_j$ and the Gaussian $G_j$,
and $F$ is the event that  at
  most $N^\prime$ of $k_1,\ldots, k_n$ are equal.

Fixing the $\Gamma_j$, applying the boundedness of the $B_{k_1\ldots
  k_n}$, taking an expectation over the Gaussian $G_j$ and using
  their symmetry, bounds the above by
\begin{equation}\label{e:1}
C\, \E\left\{ \left(\frac{1}{D} \sum_{k_1=1}^\infty 
\ldots \sum_{k_n=1}^\infty
\frac{ \Gamma_{k_1}^{-2/\alpha}
 \ldots \Gamma_{k_n}^{-2/\alpha}}{\left(\prod_{k=1}^{d}
\Gamma_k^{-2/\alpha}\Gamma_{d+1}^{-2/\alpha}
\right)^{N^\prime}} \indic_F
\right)^{1/2}\right\},
\end{equation}
where $C$ is a finite constant which may change from line to line. 
Now use the fact that $D$ is independent of all the other random variables 
in the above expectation, along with \eqref{EDfinite}, to remove the 
factor of $D^{-1}$ from the expectation, with an appropriate change of the
constant $C$. 

We may, and shall, assume
that, in the $k$-fold sum in \eqref{e:1},  $k_1\leq k_2\leq\dots\leq k_n$.
Thus what remains to show is that the remaining expectation, 
which we rewrite as
\beqq
&&\E\left\{ \left( \sum_{k_1=1}^\infty 
\ldots \sum_{k_{n-1}=k_{n-2}+1}^\infty
\frac{ \Gamma_{k_1}^{-2/\alpha}
 \ldots \Gamma_{k_{n-1}}^{-2/\alpha}}{\left(\prod_{k=1}^{d}
\Gamma_k^{-2/\alpha}\right)^{N^\prime}
\left(\Gamma_{d+1}^{-2/\alpha}\right)^{N^\prime-1}}
\right. \right.
\\ &&\qquad\qquad\qquad\qquad\qquad\qquad\qquad\qquad\times\left.\left.
\sum_{k_n=k_{n-1}+1}^\infty {\Gamma_{d+1}^{2/\alpha}}
{\Gamma_{k_n}^{-2/\alpha}}
 \indic_F
\right)^{1/2}\right\},
\eeqq
 is finite. Consider the summation over $k_n$, keeping 
$k_1,\ldots, k_{n-1}$ fixed. Since each term in
the ($n$-fold) sum is bounded from above by
1, we can  begin the summation over $k_n$ at an
arbitrary large value $K$. Then, the expectation in the last sum
is bounded by 
\beq
\label{lastterm}
\E\left\{ \left( \Gamma_{d+1}^{2/\alpha}\sum_{k_n=K}^\infty
  \Gamma_{k_n}^{-2/\alpha}\right)^{1/2}\right\}.
\eeq
Since Gamma random variables have all moments finite, by H\"older's
inequality it is  enough to prove that for  $\vep\leq 1$,
\beq
\label{e:17}
\E\left\{ \left( \sum_{k_n=K}^\infty
  \Gamma_{k_n}^{-2/\alpha}\right)^{(1+\vep)/2}\right\}<\infty.
\eeq
Furthermore, since it is easy to check that
\beqq
\E\left\{\Gamma_{k}^{-2/\alpha}\right\}\leq C k^{-2/\alpha}
\eeqq
for $k$ large enough, the finiteness of \eqref{lastterm} now follows from the
fact that the expectation in \eqref{e:17} is bounded by
\beqq
\left( \sum_{k_n=K}^\infty
\E \left\{ \Gamma_{k_n}^{-2/\alpha}\right\}\right)^{(1+\vep)/2}
\eeqq
and that $\a <2$. 

Next, we consider the double sums obtained by fixing 
$k_1,\ldots, k_{n-2}$ and taking the sum only over $k_{n-1}$ and
$k_n$. As in the case we have just considered,
we can start the sum over $k_{n-1}$
at a (large) $K$ of our choice, and then the expectation involving the
double  summation, with an additional
 factor of $\Gamma_{d+1}^{2/\a}$ brought into
play, is bounded above by 
\beqq
\E\left\{ \left(
\Gamma_{d+1}^{4/\alpha}\sum_{k_{n-1}=K}^\infty\sum_{k_{n}=K}^\infty 
 \Gamma_{k_{n-1}}^{-2/\alpha}
 \Gamma_{k_n}^{-2/\alpha}\right)^{1/2}\right\}. 
\eeqq
The finiteness of this expression follows as above, 
from the easily checkable  fact that 
\beqq
\E\left\{ \Gamma_{k_1}^{-2/\alpha}\Gamma_{k_2}^{-2/\alpha}\right\}\leq C\,
k_1^{-2/\alpha}k_2^{-2/\alpha}
\eeqq
for all $k_1\leq k_2$ large enough. Iterating this argument, we arrive
at the finiteness of \eqref{e:1}, once we prove that, for all $k_1\leq
k_2\leq \ldots\leq k_n$ large 
enough,
\beqq
\E\left\{ \prod_{j=1}^n \Gamma_{k_j}^{-2/\alpha}\right\} \leq
C\, \prod_{j=1}^n j^{-2/\alpha}.
\eeqq
This, however, follows from H\"older's inequality and the fact that
for all $k$ large enough,
\beqq
\E\left\{\Gamma_{k}^{-2n/\alpha}\right\}\leq C\, k^{-2n/\alpha}\,.
\eeqq
This completes the argument.
\qed
\end{proof}

\begin{lemma}
\label{lemma:croftontrick}
The exchange of expectation and integration in \eqref{croftontrick} is
justified.
\end{lemma}
\begin{proof}
In the notation of \eqref{croftontrick}, it will suffice to show that 
\beq
\label{app:crofton}
\int_{\graff(N,N-j)}
\E \left\{\left|\lips_0\left(A_u(f,M) \cap V\right|\right) \right \}
\, d\lambda^N_{N-j}(V) < \infty.
\eeq
Note first that \eqref{croftontrick} was only being applied under the 
conditions of one of the main theorems of the paper, so we know that the
inner expectation is always finite. We also know, from the discussions of 
this appendix that, for any convex $A$,
\beqq
\left|\lips_0\left(A_u(f,A)\right) \right| \leq \psi(A),
\eeqq
where we define $\psi (A)$ to be the number of critical points of $f$ in 
$A$ and on its various boundaries. All of these also have finite 
expectations. The functional $\E\{\psi\}$ is clearly additive (since $\psi$ 
itself is) in the sense of \eqref{G6.1.2}, as well as being invariant under
rigid motions (by isotropy) and continuous in the Hausdorff metric.
Thus, by Hadwiger's theorem (cf.\ \eqref{geometry:additivity-func:equation}) 
there exist  constants $c_j$, dependent of $u$ and the structure of $f$, such
that 
\beqq
\E\{\psi\} &=& \sum_{k=0}^N c_k \lips_k(A). 
\eeqq
Substituting into \eqref{app:crofton} we therefore have
\beqq
&&\int_{\graff(N,N-j)} 
\E \left\{\left|\lips_0\left(A_u(f,M) \cap V\right)\right| \right\}
\, d\lambda^N_{N-j}(V)\\ 
&&\qquad\qquad\qquad \leq  \int_{\graff(N,N-j)}
\E \left\{\psi(M \cap V) \right\}
\, d\lambda^N_{N-j}(V)\\
&&\qquad\qquad\qquad
= \int_{\graff(N,N-j)} \sum_{k=0}^N c_k \lips_k(M\cap V)\, d\lambda^N_{N-j}(V)\\
&&\qquad\qquad\qquad =
\sum_{k=0}^N c_k \int_{\graff(N,N-j)}  \lips_k(M\cap V)\, d\lambda^N_{N-j}(V)
\\
&&\qquad\qquad\qquad =
\sum_{j=0}^N c_j \sqbinom{k+j}{k}\lips_{k+j}(M),
\eeqq
the last line following from  Crofton's formula \eqref{crofton:formula}. 
Since $M$ is compact and convex, the  $\lips_j(M)$ are all finite, and 
so we are done.
\qed
\end{proof}

\bibliographystyle{apt}
\bibliography{grf-bib.bib,/cygdrive/c/Gena/texfiles/bibfile}

\address{Robert J.\ Adler \\
Industrial Engineering and Management\\ Technion, Haifa,
Israel 32000 \\
\printead{e1}\\
\printead{u1}}

\noindent  \address{Gennady Samorodnitsky\\
School of Operations Research\\ and Information Engineering\\
Cornell University, Ithaca, NY, 14853.\\
\printead{e2}\\
\printead{u2}}

\noindent  \address{Jonathan E.\ Taylor\\
Department of Statistics\\ Stanford University\\ Stanford, CA, 94305-4065. \\
\printead{e3}\\
\printead{u3}}

\end{document}